\documentclass[12pt]{article}
\usepackage[centertags]{amsmath}
\usepackage{amsfonts}
\usepackage{amssymb}
\usepackage{latexsym}
\usepackage{amsthm}
\usepackage{newlfont}
\usepackage{graphicx}
\usepackage{listings}
\usepackage{booktabs}
\usepackage{abstract}
\newlength{\defbaselineskip}
\newcommand{\setlinespacing}[1]%
           {\setlength{\baselineskip}{#1 \defbaselineskip}}

\newcommand{\A}{{\cal A}}

\newcommand{\D}{{\cal D}}

\newcommand{\real}{\mathbb R}

\newcommand{\spn}{\operatorname{span}}

\newcommand{\csp}{\overline{\operatorname{span}}}
\theoremstyle{plain}
\newtheorem{thm}{Theorem}[section]
\newtheorem{cor}[thm]{Corollary}
\newtheorem{lem}{Lemma}[section]
\newtheorem{prop}[thm]{Proposition}
\theoremstyle{definition}
\newtheorem{defn}{Definition}[section]

\theoremstyle{remark}

\numberwithin{equation}{section}

\begin{document}
\title{{Super Greedy Type Algorithms}\thanks{\it Math Subject Classifications.
primary: 41A65; secondary: 41A25, 41A46, 46B20.}}
\author{Entao Liu \thanks{Dept. of Mathematics, 1523 Greene Street, University of South Carolina, Columbia, SC 29208, USA e-mail: liue@mailbox.sc.edu}
and Vladimir N. Temlyakov \thanks{Dept. of Mathematics, 1523 Greene
Street,  University of South Carolina, Columbia, SC 29208, USA
e-mail: temlyak@math.sc.edu}}
\date{October 3,2010} \maketitle
\begin{abstract}
{We study greedy-type algorithms such that at a greedy step we pick several dictionary elements contrary to a single dictionary element in standard greedy-type algorithms. We call such greedy algorithms {\it super greedy algorithms}. The idea of picking several elements at a greedy step of the algorithm is not new. Recently, we observed the following new phenomenon. For incoherent dictionaries these new type of algorithms (super greedy algorithms) provide the same (in the sense of order) upper bound for the error as their  analogues from the standard greedy algorithms. The super greedy algorithms are computationally simpler than their analogues from the standard greedy algorithms. We continue to study this phenomenon.
}
\end{abstract}
\section{Introduction. Weak Super Greedy Algorithm}

This paper is a follow up to the paper \cite{LT}. We continue to study greedy-type algorithms such that at a greedy step we pick several dictionary elements contrary to a single dictionary element in standard greedy-type algorithms. We call such greedy algorithms {\it super greedy algorithms}. We refer the reader to \cite{T2} for a survey of the theory of greedy approximation.
The idea of picking several elements at a greedy step of the algorithm is not new. It was used, for instance, in \cite{T4}. A new phenomenon that we observed in \cite{LT} is the following. For incoherent dictionaries these new type of algorithms (super greedy algorithms) provide the same (in the sense of order) upper bound for the error as their  analogues from the standard greedy algorithms. The super greedy algorithms are computationally simpler than their analogues from the standard greedy algorithms. We continue to study this phenomenon here. We note that the idea of applying super greedy algorithm to incoherent dictionaries was used in \cite{KMPT} for building an efficient learning algorithm.

We  recall some notations and definitions from the theory of
greedy algorithms. Let $H$ be a real Hilbert space with an inner
product $\langle\cdot,\cdot\rangle $ and the norm $\|x\|:=\langle
x,x \rangle^{1/2}$ for all $x\in H$. We say a set $\D$ of
functions (elements) from $H$ is a dictionary if each $g\in \D$
has a unit norm $(\|g\|=1)$ and $\csp \D =H.$
 Let
$$M(\D):=\sup_{\varphi \not= \psi \atop
\varphi,\psi\in \D}| \langle \varphi, \psi \rangle|
$$
 be the
coherence parameter of dictionary $\D$. We say that a dictionary $\D$ is $M$-coherent if $M(\D)\le M$. Main results of this paper concern performance of super greedy algorithms with regard to $M$-coherent dictionaries. We study two versions of super greedy algorithms: the Weak Super Greedy Algorithm   and the Weak Orthogonal Super Greedy Algorithm with
Thresholding. We now proceed to the definitions of these algorithms and to the formulations of main results.

 Let a natural number $s$ and a weakness sequence
$\tau:=\{t_k\}_{k=1}^\infty$, $t_k\in [0,1]$, be given. Consider
the following Weak Super Greedy Algorithm with
parameter $s$.

{\bf WSGA($s,\tau$).} Initialization: $f_0:=f^{s,\tau}_0:=f$. Then for
each $m\ge 1$ we inductively define:
\begin{itemize}
\item[(1)] $\varphi_{(m-1)s+1},\dots,\varphi_{ms}\in \D$ are elements of
the dictionary $\D$ satisfying the following inequality. Denote
$I_m:=[(m-1)s+1,ms]$ and assume that
$$
\min_{i\in I_m}|\langle f_{m-1},\varphi_i\rangle |\ge
t_m\sup_{g\in \D, g\neq \varphi_i, i\in I_m} |\langle
f_{m-1},g\rangle |.
$$

\item[(2)] Let $F_m:=F_m(f_{m-1}):=\spn(\varphi_i, i\in I_m)$ and let
$P_{F_m}$ denote an operator of orthogonal projection onto $F_m$.
Define the residual after $m$th iteration of the algorithm
$$
f_m:= f^{s,\tau}_m:= f_{m-1}-P_{F_m}(f_{m-1}).
$$

\item[(3)] Find the approximant
$$
G^s_m(f):=G^{s,\tau}_m(f,\D):= \sum_{j=1}^m P_{F_j}(f_{j-1}).
$$
\end{itemize}
In the case $t_k=t$, $k=1,2,\ldots$, we write $t$ instead of
$\tau$ in the notations. If $t=1$, we call the WSGA$(s,1)$ the
Super Greedy Algorithm with parameter $s$ (SGA$(s)$).
For $s=1$ the Super Greedy Algorithm coincides with the Pure Greedy Algorithm and the Weak Super Greedy Algorithm coincides with the Weak Greedy Algorithm (see \cite{T2}).

For a general dictionary $\D$ we define the class of functions (elements)
\begin{equation*}
\A_1^0(\D,B):=\left\{f\in H: f=\sum_{k\in \Lambda} c_k g_k, \quad
g_k\in \D, |\Lambda|< \infty, \quad \sum_{k\in \Lambda}|c_k|\leq
B\right\}
\end{equation*}
and we define $\A_1(\D,B)$ to be the closure (in $H$) of
$\A^0_1(\D,B)$. For the case $B=1$, we denote
$A_1(\D):=\A_1(\D,1)$. We   define the  norm
$
|f|_{\A_1(\D)}
$
to be the smallest $B$ such that $f\in \A_1(\D,B)$.

The following open problem (see \cite{T5}, p.65, Open Problem 3.1) on the rate of convergence of the PGA for the $A_1(\D)$  is a central theoretical problem in greedy approximation in Hilbert spaces.

{\bf Open problem.} Find the order of decay of the sequence
$$
\gamma(m):=\sup_{f,\D,\{G^1_m\}} (\|f-G^1_m(f,\D)\| |f|^{-1}_{\A_1(\D)}),
$$
where the supremum is taken over all dictionaries $\D$, all elements
$f\in \A_1(\D)\setminus\{0\}$ and all possible choices of $\{G^1_m\}$.

We refer the reader to \cite{T2} for a discussion of this open problem.
 Introduce the following generalization of the quantity $\gamma(m)$ to the case of the Weak Greedy Algorithms
$$
\gamma(m,\tau):=\sup_{f,\D,\{G^{1,\tau}_m\}} (\|f-G^{1,\tau}_m(f,\D)\| |f|^{-1}_{\A_1(\D)}).
$$

We prove here the following theorem.

\begin{thm}\label{wsga}
Let $\D$ be a dictionary with coherence parameter $M:=M(\D)$. Then,
for $s\le (2M)^{-1}$, the WSGA($s,t$) provides, after $m$ iterations,
an approximation of $f\in A_1(\D)$ with the following upper bound
on the error:
\begin{equation*}
\|f-G^{s,t}_m(f,\D)\|\leq Ct^{-1} s^{-\frac{1}{2}}\gamma(m-1,rt),
\end{equation*}
where $r:=\Big(\frac{1-Ms}{1+Ms}\Big)^{\frac{1}{2}}$ and $C$ is an
absolute constant.
\end{thm}

 Theorem \ref{wsga} with $t=1$ gives the
following assertion for the SGA($s$).
\begin{cor}\label{sga}
Let $\D$ be a dictionary with coherence parameter $M:=M(\D)$. Then,
for $s\le (2M)^{-1}$, the SGA($s$) provides, after $m$ iterations, an
approximation of $f\in A_1(\D)$ with the following upper bound on
the error:
\begin{equation*}
\|f-G^s_m(f,\D)\|\leq C s^{-\frac{1}{2}}\gamma(m,r),
\end{equation*}
where $r:=\Big(\frac{1-Ms}{1+Ms}\Big)^{\frac{1}{2}}$ and $C$ is an
absolute constant.
\end{cor}

It is interesting to note that even in the case of the SGA($s$), when the weakness parameter is $1$, we have the upper bound of the error in terms of
$\gamma(m,r)$ not in terms of $\gamma(m)$. For estimating $\gamma(m,r)$ we use the following known result  from \cite{T1}.

\begin{thm}\label{wga}
Let $\D$ be an arbitrary dictionary in $H$. Assume
$\tau:=\{t_k\}^\infty_{k=1}$ is a nonincreasing sequence. Then, for
$f\in \A_1(D,B)$ we have
\begin{equation}
\|f-G^{1,\tau}_m(f,\D)\|\leq B(1+\sum_{k=1}^m t_k^2)^{-t_m/2(2+t_m)}.
\end{equation}
\end{thm}

For a particular case $t_k=1$, $k=1,2,\dots$, this theorem gives the following result (see \cite{DT}).
For each $f\in
\A_1(\D,B)$, the PGA provides, after $m$ iterations, an approximant   satisfying
$$
\|f-G_m(f,\D)\|\leq B m^{-1/6}.
$$

For $f\in A_1(\D)$, we apply the PGA and the SGA($s$). Then after $sm$
iterations of the PGA and $m$ iterations of the SGA($s$), both
algorithms provide $sm$-term approximants. For illustration purposes, take  $s=N^{1/2}$ and
$m=N^{1/2}$. Then the PGA gives
$$
\|f_N\|\leq (sm)^{-1/6}=N^{-1/6}
$$
and the SGA($s$) provides
$$
\|f_m\|\leq
Cs^{-1/2}m^{-\frac{r}{2(2+r)}}=N^{-\frac{1}{4}-\frac{1}{2}\theta},
$$
where $\theta:=\frac{r}{2(2+r)}$ and
$r=\Big(\frac{1-Ms}{1+Ms}\Big)^{1/2}$,
$
\frac{1}{4}+\frac{1}{2}\theta\geq \frac{1}{6}.
$
 Thus, in this particular case, the SGA($s$) has a better upper bound for the error than the PGA.

\section{Weak Orthogonal Super Greedy Algorithm with Thresholding}

In \cite{LT} we considered the following algorithm. Let a natural number $s$ and a weakness sequence
$\tau:=\{t_k\}_{k=1}^\infty$, $t_k\in [0,1]$, be given. Consider
the following Weak Orthogonal Super Greedy Algorithm with
parameter $s$.

{\bf WOSGA($s,\tau$).} Initially, $f_0:=f$. Then, for each $m\ge 1$
we inductively define:

(1) $\varphi_{(m-1)s+1},\dots,\varphi_{ms}\in \D$ are elements of
the dictionary $\D$ satisfying the following inequality. Denote
$I_m:=[(m-1)s+1,ms]$ and assume that
$$
\min_{i\in I_m}|\langle f_{m-1},\varphi_i\rangle |\ge
t_m\sup_{g\in \D, g\neq \varphi_i, i\in I_m} |\langle
f_{m-1},g\rangle |.
$$

(2) Let $H_m:=H_m(f):=\spn(\varphi_1,\dots,\varphi_{ms})$ and let
$P_{H_m}$ denote an operator of orthogonal projection onto $H_m$.
Define
$$
G_m(f):=G_m(f,\D):= G_m^s(f,\D):= P_{H_m}(f).
$$

(3) Define the residual after $m$th iteration of the algorithm
$$
f_m:= f^s_m:= f-G_m(f,\D).
$$

 In \cite{LT} we proved the following error bound for the WOSGA($s,t$).

\begin{thm}\label{Theorem 28.3}
Let $\D$ be a dictionary with coherence parameter $M:=M(\D)$. Then,
for $s\le (2M)^{-1}$, the WOSGA($s,t$) provides, after $m$
iterations, an approximation of $f\in A_1(\D)$ with the following
upper bound on the error:
$$
\|f_m\|^2 \le A(t)(sm)^{-1},\quad m=1,2,\dots\quad
A(t):=(81/8)(1+t)^2t^{-4}.
$$
\end{thm}

In this paper we modify the WOSGA in the following way: we replace the greedy step (1) by the thresholding step. Here is the definition of the new algorithm.
 Let $s$ be a natural number and let a weakness sequence
$\tau:=\{t_k\}_{k=1}^\infty$, $t_k\in [0,1]$ be given.

{\bf WOSGAT($s,\tau$).} Initially, $f_0:=f^{s,\tau}_0:=f$. Then for
each $m\ge 1$ we inductively define:
\begin{itemize}
\item[(1)] $\varphi_{i}\in \D$ where $i_m \in I_m$ are elements of
the dictionary $\D$ satisfying the following inequalities
$s_m:=|I_m|\le s$ and
\begin{equation}\label{gr}
\min_{i\in I_m}|\langle f_{m-1},\varphi_i\rangle |\ge
t_m\|f_{m-1}\|^2.
\end{equation}

\item[(2)] Let $H_m:=H_m(f):=\spn(\varphi_i, \  i\in I_1 \cup \ldots \cup I_m)$ and let
$P_{H_m}$ denote an operator of orthogonal projection onto $H_m$.
Denote
$$
G_m(f):=G^{s,\tau}_m(f,\D):= P_{H_m}(f).
$$

\item[(3)] Define the residual after $m$th iteration of the algorithm
$$
f_m:= f^{s,\tau}_m:= f -P_{H_m}(f).
$$
\end{itemize}

For $s=1$ the WOSGA coincides with the Weak Orthogonal Greedy Algorithm (WOGA) and the WOSGAT coincides with the Modified Weak Orthogonal Greedy Algorithm (MWOGA) (see \cite{T5}, p. 61). We note that we can run the WOSGA and the WOGA for any $f\in H$. It is proved in \cite{T5} that we can run the MWOGA for $f\in A_1(\D)$. In the same way one can prove that we can run the WOSGAT for $f\in A_1(\D)$.
We note that in step (1) of the WOSGAT, if there are more than $s$ $\varphi_i$'s satisfying
$$
|\langle f_{m-1},\varphi_i\rangle |\ge
t_m\|f_{m-1}\|^2,
$$
the algorithm may pick any $s$ of them and then make the projection.

If $t_m=t$ for $m=1,2,\ldots$, we  use $t$ instead of $\tau$ in the notation.
We will prove an upper bound for the rate of convergence of the WOSGAT for $f\in A_1(\D)$ for a more general dictionary than the $M$-coherent dictionary.
\begin{defn}\label{Bessel}
We say that a dictionary $\D$ is $(N,\beta)$-Bessel if for any $N$ distinct elements $\psi_1,\dots,\psi_N$ of the dictionary $\D$ we have for any $f\in H$
$$
\|P_{\Psi(N)}(f)\|^2 \ge \beta\sum_{i=1}^N|\langle f, \psi_i\rangle|^2,
$$
where $\Psi(N):=\spn\{\psi_1,\dots,\psi_N\}$.
\end{defn}
\begin{thm}\label{OGAT}
Let $\D$ be an $(N,\beta)$-Bessel dictionary. Then, for $s\le N$,
the WOSGAT($s,t$) provides, after $m$ iterations, an approximation
of $f\in A_1(\D)$ with the following upper bound of the error:
$$
\|f-G_m(f)\|\le  (1+\beta t^2\sum_{j=1}^m s_j)^{-1/2}.
$$
\end{thm}
We point out that $\sum_{j=1}^m s_j$ is the number of elements that the algorithm picked up
from $\D$ after $m$ iterations. Therefore, the WOSGAT offers the same error bound (in the sense of order) in terms of the number of elements of the dictionary, used in the approximant, as the
WOGA or the MWOGA.

We now give some sufficient conditions for a dictionary $\D$ to be $(N,\beta)$-Bessel. We begin with a simple lemma useful in that regard.
\begin{lem}\label{L2.1} Let a dictionary $\D$ have the following property of $(N,A)$-stability. For any $N$ distinct elements $\psi_1,\dots,\psi_N$ of the dictionary $\D$, we have for any coefficients $c_1,\dots,c_N$
$$
\|\sum_{i=1}^Nc_i\psi_i\|^2\le A\sum_{i=1}^N|c_i|^2.
$$
Then $\D$ is $(N,A^{-1})$-Bessel.
\end{lem}
\begin{proof} Let $f\in H$. We have
$$
\|P_{\Psi(N)}(f)\| = \sup_{\psi\in\Psi(N), \|\psi\|\le 1}|\langle P_{\Psi(N)}(f),\psi\rangle|
$$
$$
=\sup_{(c_1,\dots,c_N):\|c_1\psi_1+\dots+c_N\psi_N\|\le 1}|\sum_{i=1}^N\langle f,\psi_i\rangle c_i|
$$
$$
\ge \sup_{(c_1,\dots,c_N):|c_1|^2+\dots+|c_N|^2\le A^{-1}}|\sum_{i=1}^N\langle f,\psi_i\rangle c_i| = A^{-1/2}(\sum_{i=1}^N|\langle f,\psi_i\rangle|^2)^{1/2}.
$$
\end{proof}

\begin{prop}\label{prop1} An $M$-coherent dictionary is $(N,(1+M(N-1))^{-1})$-Bessel.
\end{prop}
\begin{proof} By Lemma 2.1 from \cite{DET} for an $M$-coherent dictionary we have
$$
\|\sum_{i=1}^N c_i\psi_i\|^2 \le (1+M(N-1))\sum_{i=1}^N |c_i|^2.
$$
Applying the above Lemma \ref{L2.1}, we obtain the statement of the proposition.
\end{proof}

The following proposition is a direct corollary of Lemma \ref{L2.1}.

\begin{prop}\label{prop2} Let a dictionary $\D$ have the RIP$(N,\delta)$: for any distinct $\psi_1,\dots,\psi_N$
$$
(1-\delta)\sum_{i=1}^N |c_i|^2\le  \|\sum_{i=1}^N c_i\psi_i\|^2\le (1+\delta)\sum_{i=1}^N |c_i|^2.
$$
Then $\D$ is $(N,(1+\delta)^{-1})$-Bessel.
\end{prop}

\section{Proofs}

\noindent{\bf Proof of Theorem \ref{wsga}.}  Let
\begin{equation}\label{4.2}
f=\sum_{j=1}^\infty c_jg_j,\quad g_j\in \D,\quad \sum_{j=1}^\infty
|c_j|\le 1,\quad |c_1|\ge |c_2|\ge \dots\quad .
\end{equation}
Every element of $A_1(\D)$ can be approximated arbitrarily well by
elements of the form (\ref{4.2}). It will be clear from the below
argument that it is sufficient to consider elements $f$ of the
form (\ref{4.2}). Suppose $\nu$ is such that $|c_\nu|\ge a/s\ge
|c_{\nu+1}|$, where $a=\frac{t+3}{2t}$. Then the above
assumption on the sequence $\{c_j\}$ implies that $\nu\le \lfloor
s/a \rfloor$ and $|c_{s+1}|<1/s$. We claim that elements
$g_1,\dots,g_\nu$ will be chosen among $\varphi_1,\dots,\varphi_s$
at the first iteration. Indeed, for $j\in [1,\nu]$ we have
$$
|\langle f,g_j\rangle | \ge |c_j|-M\sum_{k\neq j}^\infty |c_k|\ge
a/s -M(1-a/s)>a/s - M.
$$
For all $g$ distinct from $g_1,\dots,g_s$ we have
$$
|\langle f,g\rangle | \le M+1/s.
$$
Our assumption $s\le 1/(2M)$ implies that $M+1/s \le t(a/s -M)$.
Thus, we do not pick any of $g\in\D$ distinct from $g_1,\dots,g_s$
until we have chosen all $g_1,\dots,g_\nu$.

Denote
$$
f':= f-\sum_{j=1}^\nu c_jg_j = \sum_{j=\nu+1}^\infty c_jg_j.
$$
It is clear from the above argument that
\begin{equation}\label{f1}
f_1=f-P_{H_1}(f) = f'-P_{H_1}(f')=f'-G_1^s(f').
\end{equation}
Define a new dictionary
$$
\D^s:=\left\{  \frac{\sum_{j\in \Lambda} c_j g_j}{\|\sum_{j\in
\Lambda} c_j g_j\|}: \quad |\Lambda |=s,  \quad g_j\in \D, c_j\in
\real \right\}.
$$
Let $J_l=[(l-1)s+\nu+1,ls+\nu]$. We write
\begin{equation}\label{4.4}
f'=\sum^{\infty}_{l=1}\sum_{j \in J_l}c_j g_j
=\sum^{\infty}_{l=1}\|\psi_l\|\frac{\psi_l}{\|\psi_l\|},
\end{equation}
where $\psi_l=\sum_{j\in J_l}c_j g_j$. Apparently
$\frac{\psi_l}{\|\psi_l\|} \in \D^s$ for $l\geq 1$. Equation
(\ref{4.4}) implies that
$$
f'\in \A_1(\D^s,\sum^{\infty}_{l=1}\|\psi_l\|).
$$
By Lemma 2.1 from
\cite{DET} we bound
\begin{equation}\label{4.5}
(1-Ms)\sum_{j\in J_l} c_j^2 \le \|\psi_l \|^2 \le (1+Ms)\sum_{j
\in J_l} c_j^2.
\end{equation}
Then we obtain
\begin{equation}\label{4.6}
\sum^{\infty}_{l=1}\|\psi_l\|\leq (1+Ms)^{1/2}\sum_{l=1}^\infty
(\sum_{j\in J_l}c^2_j)^{1/2}.
\end{equation}
Since the sequence $\{c_j\}$ has the property
\begin{equation}
|c_{\nu+1}|\ge |c_{\nu+2}|\ge \dots , \quad \sum_{j=\nu+1}^\infty
|c_j|\le 1,\quad |c_{\nu+1}|\le a/s
\end{equation}
we may apply the simple inequality,
$$
(\sum_{j\in J_l} c_j^2)^{1/2}\le s^{1/2}|c_{(l-1)s+\nu+1}|,
$$
so that we bound the sum in the right side of (\ref{4.6})
\begin{eqnarray}\label{4.8}
\sum_{l=1}^\infty(\sum_{j\in J_l} c_j^2)^{1/2} &\le&
s^{1/2}\sum_{l=1}^\infty|c_{(l-1)s+\nu+1}|\nonumber\\
 &\le& s^{1/2}(a/s + \sum_{l=2}^\infty s^{-1}\sum_{j\in J_{l-1}}|c_j|) \le (a+1)s^{-1/2}.
\end{eqnarray}
Using the above inequality in (\ref{4.6}), we obtain that
\begin{equation}\label{4.9}
f'\in \A_1(\D^s,(3/2)^{1/2}(a+1)s^{-\frac{1}{2}}).
\end{equation}

Assume $u:=\sum a_i h_i \in \A_1(\D^s,B)$, where $B$ is an
absolute constant. Then for any $\psi\in \D^s$ we have
\begin{equation*}
v:=u-\langle u,\psi \rangle \psi \in \A_1(\D,2B),
\end{equation*}
since $|\langle u,\psi \rangle|=|\langle \sum a_i g_i, \psi\rangle
|\leq \|\psi\|\sum|a_i|\leq B$. Along with (\ref{f1}) and
(\ref{4.9}) the above argument shows that
\begin{equation}\label{4.10}
f_1\in \A_1(\D^s,2(3/2)^{1/2}(a+1)s^{-1/2}).
\end{equation}

Consider the following quantity
\begin{equation*}
q_s := q_s(f_{m-1}):= \sup_{h_i\in\D  \atop
i\in[1,s]}\|P_{H(s)}(f_{m-1})\|,
\end{equation*}
where $H(s):=\spn(h_1,\dots,h_s)$. It is clear that
$$ q_s=\sup_{H(s)}\max_{\psi\in H(s), \|\psi\|\le 1}|\langle
f_{m-1},\psi\rangle |= \sup_{g^s\in \D^s}|\langle f_{m-1},g^s \rangle|  .
$$
Let $\psi = \sum_{i=1}^s a_ih_i$. Again by Lemma 2.1 from
\cite{DET} we bound
\begin{equation}\label{4.11}
(1-Ms)\sum_{i=1}^s a_i^2 \le \|\psi\|^2 \le (1+Ms)\sum_{i=1}^s
a_i^2.
\end{equation}
Therefore,
\begin{equation}\label{4.12}
(1+Ms)^{-1} \sum_{i=1}^s \langle f_{m-1},h_i\rangle ^2 \le
\|P_{H(s)}(f_{m-1})\|^2\le (1-Ms)^{-1} \sum_{i=1}^s \langle
f_{m-1},h_i\rangle ^2.
\end{equation}
Let  $p_m:= P_{H_m}(f_{m-1})$. In order to relate $q_s^2$ to $\|p_m\|^2$ we begin with the fact that
$$
q_s^2 \le
\sup_{h_i\in\D \atop i\in[1,s]}(1-Ms)^{-1}\sum_{i=1}^s\langle f_{m-1},h_i\rangle ^2.
$$

Consider an arbitrary set $\{h_i\}_{i=1}^s$ of distinct elements
of the dictionary $\D$. Let $V$ be a set of all indices
$i\in[1,s]$ such that $h_i=\varphi_{k(i)}$, $k(i)\in I_m$. Denote
$V':=\{k(i), i\in V\}$. Then
\begin{equation}\label{28.11}
\sum_{i=1}^s\langle f_{m-1},h_i\rangle ^2 = \sum_{i\in V}\langle
f_{m-1},h_i\rangle ^2 + \sum_{i\in [1,s]\setminus V}\langle
f_{m-1},h_i\rangle ^2.
\end{equation}
From the definition of $\{\varphi_k\}_{k\in I_m}$ we get
\begin{equation}\label{28.12}
\max_{i\in [1,s]\setminus V}|\langle f_{m-1},h_i\rangle | \le
t^{-1}\min_{k\in I_m\setminus V'}|\langle f_{m-1},\varphi_k\rangle
|.
\end{equation}
Using (\ref{28.12}) we continue (\ref{28.11})
$$
\le \sum_{k\in V'}\langle f_{m-1},\varphi_k\rangle ^2 + t^{-2}\sum_{k\in
I_m\setminus V'}\langle f_{m-1},\varphi_k\rangle ^2 \le t^{-2}\sum_{k\in
I_m}\langle f_{m-1},\varphi_k\rangle ^2.
$$
Therefore,
$$
q_s^2 \le (1-Ms)^{-1}t^{-2}\sum_{k\in I_m}\langle f_{m-1},\varphi_k\rangle ^2\le
\frac{1+Ms}{t^2(1-Ms)} \|p_m\|^2.
$$
This results in the following inequality
\begin{equation}\label{28.13}
\|p_m\|^2 \ge \frac{t^2(1-Ms)}{1+Ms} q_s^2.
\end{equation}

Thus we can interpret WSGA($s,t$) as WGA($rt$) with respect to the
dictionary $\D^s$, where
$r=\Big(\frac{1-Ms}{1+Ms}\Big)^{\frac{1}{2}}$.  Using (\ref{4.10}), we get
$$
\|f^{s,t}_m\| = \|(f_1)^{s,t}_{m-1}\| \le 6^{1/2}(a+1)s^{-1/2}\gamma(m-1,rt).
$$
This completes the proof of Theorem \ref{wsga}.

\noindent{\bf Proof of Theorem \ref{OGAT}.} For the $\{\varphi_i\}$ from the definition of the WOSGAT, denote
$$
F_m := \spn(\varphi_i, i\in I_m).
$$
It is easy to see that $H_m=H_{m-1}\oplus F_m$. Therefore,
\begin{eqnarray}
f_m = f-P_{H_m}(f) &=& f_{m-1}+G_{m-1}(f) -
P_{H_m}(f_{m-1}+G_{m-1}(f))
\nonumber\\
&=&
f_{m-1}-P_{H_m}(f_{m-1}).\nonumber
\end{eqnarray}
Then the inclusion $F_m\subset H_m$ implies
\begin{equation}\label{2.1}
\|f_m\| \le \|f_{m-1}-P_{F_m}(f_{m-1})\|.
\end{equation}
Using the notation $p_m:= P_{F_m}(f_{m-1})$, we continue
$$
\|f_{m-1}\|^2 =\|f_{m-1}-p_m\|^2 + \|p_m\|^2
$$
and by (\ref{2.1})
\begin{equation}\label{2.2}
 \|f_m\|^2  \le \|f_{m-1}\|^2 -\|p_m\|^2 .
\end{equation}

We now prove a lower bound for $\|p_m\|$.
 By our assumption that the dictionary is $(N,\beta)$-Bessel we get
$$
\|P_{F_m}(f_{m-1})\|^2 \ge \beta\sum_{i\in I_m}|\langle f_{m-1},\varphi_i\rangle|^2 .
$$
 Then, by the thresholding condition of the greedy step (1),   we obtain
 \begin{equation}
\|P_{F_m}(f_{m-1})\|^2  \ge \beta t^2 s_m \|f_{m-1}\|^4.
\end{equation}
Substituting this bound in (\ref{2.2}), we get
\begin{equation}
\|f_m\|^2 \le \|f_{m-1}\|^2\Big(1-\beta t^2 s_m \|f_{m-1}\|^2\Big).
\end{equation}
We now apply the  following  lemma from \cite{T1}.
\begin{lem}\label{l1}
Let $\{a_m\}^\infty_{m=0}$ be a sequence of nonnegative numbers satisfying the inequalities
$$
a_0\le A, \quad a_m\le a_{m-1}(1-\lambda^2_m a_{m-1}/A), \quad m=1,2,\ldots
$$
Then we have for each $m$
$$
a_m\le A(1+\sum_{k=1}^m \lambda^2_k)^{-1}.
$$
\end{lem}
It gives us
$$
\|f_m\|\le  (1+\beta t^2\sum_{j=1}^m s_j)^{-1/2}.
$$
\hfill $\Box$

{\bf Acknowledgements.} We are grateful to Prof. Ming-Jun Lai for helpful discussions.
This research was supported by the National Science Foundation Grant DMS-0906260.

\end{document}